\documentclass[a4paper]{amsart}

\usepackage{amsmath,amsthm,amssymb,latexsym,epic}
\usepackage{graphicx,enumerate,stmaryrd}
\newtheorem{theorem}{Theorem}
\newtheorem{lemma}[theorem]{Lemma}

\newtheorem{corollary}[theorem]{Corollary}
\newtheorem{proposition}[theorem]{Proposition}

\usepackage[all]{xy}

\begin{document}
\newcommand{\tto}{\twoheadrightarrow}

\title{Blocks of the category of cuspidal $\mathfrak{sp}_{2n}$-modules}
\author{Volodymyr Mazorchuk and Catharina Stroppel}
\date{\today}

\maketitle

\begin{abstract}
We show that every block of the category of cuspidal generalized
weight modules with finite dimensional generalized weight spaces over
the Lie algebra $\mathfrak{sp}_{2n}(\mathbb{C})$ is equivalent to the
category of finite dimensional $\mathbb{C}[[t_1,t_2,\dots,t_n]]$-modules.
\end{abstract}


\section{Introduction and description of the results}\label{s0}

We fix the ground field to be the complex numbers.
Fix $n\in\{2,3,\dots\}$ and consider the symplectic Lie algebra
$\mathfrak{sp}_{2n}=:\mathfrak{g}$ with a fixed Cartan subalgebra
$\mathfrak{h}$ and root space decomposition
\begin{displaymath}
\mathfrak{g}=\mathfrak{h} \oplus
\bigoplus_{\alpha\in\Delta}\mathfrak{g}_{\alpha},
\end{displaymath}
where $\Delta$ denotes the  corresponding set of roots.
For a $\mathfrak{g}$-module $V$ and
$\lambda\in\mathfrak{h}^*$ set
\begin{gather*}
V_{\lambda}:=\{v\in V:h\cdot v=\lambda(h)
v\text{ for any }h\in\mathfrak{h}\},\\
V^{\lambda}:=\{v\in V:(h-\lambda(h))^k\cdot v=0
\text{ for any }h \in\mathfrak{h}\text{ and }k\gg 0\}.
\end{gather*}
A $\mathfrak{g}$-module $V$ is called
\begin{itemize}
\item {\em weight} provided that
$V=\oplus_{\lambda\in\mathfrak{h}^*}V_{\lambda}$;
\item {\em generalized weight} provided that
$V=\oplus_{\lambda\in\mathfrak{h}^*}V^{\lambda}$;
\item {\em cuspidal} provided that for any $\alpha\in\Delta$ the action
of any nonzero element from $\mathfrak{g}_{\alpha}$ on $V$ is bijective.
\end{itemize}
If $V$ is a generalized weight module, then the
set $\{\lambda\in\mathfrak{h}^*:V_{\lambda}\neq 0\}$ is called
the {\em support} of $V$ and is denoted by $\mathrm{supp}(V)$.

Denote by $\hat{\mathcal{C}}$  the full subcategory in
$\mathfrak{g}\text{-}\mathrm{mod}$ which consist of all cuspidal
generalized weight modules with finite-dimensional generalized weight
spaces, and by $\mathcal{C}$ the full subcategory of
$\hat{\mathcal{C}}$ consisting of all weight modules.
Understanding the categories $\mathcal{C}$ and $\hat{\mathcal{C}}$
is a classical problem in the representation theory of Lie algebras.
The first major step towards the solution of this problem was made in
\cite{Ma}, where all simple objects in $\hat{\mathcal{C}}$ were classified.
In \cite{BKLM} it was shown that the category $\mathcal{C}$ is semi-simple,
hence completely understood. The aim of the present note is to describe
the category $\hat{\mathcal{C}}$.

Apart from $\mathfrak{sp}_{2n}$, cuspidal weight modules with
finite dimensional weight spaces exist only for the Lie algebra
$\mathfrak{sl}_{n}$ (\cite{Fe}). In the latter case, simple objects in
the corresponding category $\hat{\mathcal{C}}$ are classified in
\cite{Ma}, the category $\mathcal{C}$ is described in \cite{GS},
see also \cite{MS},
and the category $\hat{\mathcal{C}}$ is described in \cite{MS}.
Taking all these results into account, the present paper completes
the study of cuspidal generalized weight modules with finite dimensional
generalized weight spaces over semi-simple finite-dimensional
Lie algebras.

Let $U(\mathfrak{g})$ be the universal enveloping algebra of $\mathfrak{g}$
and $Z(\mathfrak{g})$ be the center of $U(\mathfrak{g})$. The action of
$Z(\mathfrak{g})$ on any object from $\hat{\mathcal{C}}$ is locally finite.
Using this and the standard support arguments gives the following
{\em block decomposition} of $\hat{\mathcal{C}}$:
\begin{displaymath}
\hat{\mathcal{C}}\cong  \bigoplus_{\text{\tiny\begin{tabular}{c}
$\chi:Z(\mathfrak{g})\to\mathbb{C}$\\$\xi\in\mathfrak{h}^*/
\mathbb{Z}\Delta$\end{tabular}}}\hat{\mathcal{C}}_{\chi,\xi},
\end{displaymath}
where $\hat{\mathcal{C}}_{\chi,\xi}$ consists of all $V$ such that
$\mathrm{Supp}(V)\subset \xi$ and $(z-\chi(z))^k\cdot v=0$ for all
$v\in V$, $z\in Z(\mathfrak{g})$ and  $k\gg 0$.
Set $\mathcal{C}_{\chi,\xi}:=\mathcal{C}\cap\hat{\mathcal{C}}_{\chi,\xi}$.
From \cite[Section~9]{Ma} it follows that each nontrivial
$\hat{\mathcal{C}}_{\chi,\xi}$ contains a unique (up to isomorphism) simple
object, in particular, $\hat{\mathcal{C}}_{\chi,\xi}$ is indecomposable,
hence a block.
From this and \cite{BKLM} we thus get that every nontrivial
block $\mathcal{C}_{\chi,\xi}$ is equivalent to the category of
finite-dimensional
$\mathbb{C}$-modules. Our main result is the following:

\begin{theorem}\label{thm1}
Every nontrivial block $\hat{\mathcal{C}}_{\chi,\xi}$ is equivalent
to the category of finite dimensional
$\mathbb{C}[[t_1,t_2,\dots,t_n]]$-modules.
\end{theorem}

To prove Theorem~\ref{thm1} we use and further develop the technique of
extension of the module structure from a Lie subalgebra, originally
developed  in \cite{MS} for the study of categories of singular and
non-integral cuspidal generalized weight $\mathfrak{sl}_{n}$-modules.
The proof of Theorem~\ref{thm1} is given in Section~\ref{s3}. In
Section~\ref{s1} we recall the standard reduction to the case of
the so-called simple {\em completely pointed} modules (i.e. simple weight
cuspidal modules for which {\bf all} nontrivial weight spaces are
one-dimensional) and a realization of such modules using differential
operators. In Section~\ref{s2} we define a functor from the category of
finite dimensional $\mathbb{C}[[t_1,t_2,\dots,t_n]]$-modules to any block
$\hat{\mathcal{C}}_{\chi,\xi}$ containing a simple completely pointed
module. In  Section~\ref{s3} we prove that this functor is an equivalence
of categories. In Section~\ref{s4} we present some consequences of
our main result, in particular, we recover the main result from
\cite{BKLM} stated above.
\vspace{5mm}

\noindent
{\bf Acknowledgments.} The first author was partially supported by the
Royal Swedish Academy of Sciences and the Swedish  Research Council.
The paper was written up, when the second author was visiting the Mathematical Sciences Research Institute (MSRI) in Berkeley. She deeply acknowledges the excellent working conditions and financial support.

\section{Completely pointed simple cuspidal weight modules}\label{s1}

A weight $\mathfrak{g}$-module $V$ is called {\em pointed} provided that
$\dim V_{\lambda}=1$ for some $\lambda\in\mathfrak{h}^*$. If $V$ is a
pointed simple cuspidal weight $\mathfrak{g}$-module, then, obviously,
all nontrivial weight spaces of $V$ are one-dimensional, in which case
one says that $V$ is {\em completely pointed} (see \cite{BKLM}).
It is enough to consider blocks with completely pointed simple
modules because of the following:

\begin{lemma}\label{lem2}
All nontrivial blocks of $\hat{\mathcal{C}}$ are equivalent.
\end{lemma}

\begin{proof}
In the case of the category $\mathcal{C}$ this is proved in
\cite[Lemma~2]{BKLM}. The same argument works in the case of the
category $\hat{\mathcal{C}}$ as well.
\end{proof}

Let us recall the explicit realization of completely pointed simple
cuspidal modules from \cite{BL}. Denote by $W_n$ the $n$-th
{\em Weyl algebra}, that is the algebra of differential operators
with polynomial coefficients in variables $x_1,x_2,\dots,x_n$.
The algebra $W_n$ is generated by $x_i$ and
$\frac{\partial}{\partial x_i}$, $i=1,\dots,n$, which satisfy the
relations $[\frac{\partial}{\partial x_i},x_j]=\delta_{i,j}$.
Let $\varepsilon_1,\varepsilon_2,\dots,\varepsilon_n$ be the vectors
of the standard basis in $\mathbb{C}^n$. Identify $\mathbb{C}^n$
with $\mathfrak{h}^*$ such that $\Delta$ becomes the following
{\em standard root system} of type $C_n$:
\begin{displaymath}
\{\pm(\varepsilon_i\pm \varepsilon_j):1\leq i<j\leq n\}
\cup\{\pm 2\varepsilon_i: 1\leq i\leq n\}.
\end{displaymath}
Then
\begin{displaymath}
\mathbf{H}=\mathbf{H}_n=\{2\varepsilon_1,\varepsilon_2-\varepsilon_{1},
\varepsilon_3-\varepsilon_{2},\dots,
\varepsilon_{n}-\varepsilon_{n-1}\}
\end{displaymath}
is a basis of $\Delta$. Fix a basis of $\mathfrak{g}$ of the form
\begin{displaymath}
\mathbf{C}:=\{X_{\pm\varepsilon_i\pm \varepsilon_j}:1\leq i<j\leq n\}
\cup\{X_{\pm 2\varepsilon_i}:i=1,2,\dots,n\}\cup
\{H_{\alpha}:\alpha\in \mathbf{H}\}
\end{displaymath}
such that the following map defines an injective Lie algebra homomorphism
from $\mathfrak{g}$ to the Lie algebra associated with $W_n$:
\begin{equation}\label{eq00}
\begin{array}{lcll}
X_{\varepsilon_i-\varepsilon_j}&\mapsto&
x_i\frac{\partial}{\partial x_j}, & 1\leq i\neq j\leq n;\\
X_{\varepsilon_i+\varepsilon_j}&\mapsto&
x_ix_j, & i,j=1,2,\dots,n;\\
X_{-\varepsilon_i-\varepsilon_j}&\mapsto&
\frac{\partial}{\partial x_i}\frac{\partial}{\partial x_j},
& i,j=1,2,\dots,n;\\
H_{\varepsilon_{i+1}-\varepsilon_{i}}&\mapsto&
x_{i+1}\frac{\partial}{\partial x_{i+1}}-x_{i}\frac{\partial}{\partial x_{i}},
& i=1,2,\dots,n-1;\\
H_{2\varepsilon_{1}}&\mapsto&
\frac{1}{2}\left(x_1\frac{\partial}{\partial x_1}+
\frac{\partial}{\partial x_{1}}x_1\right).
\end{array}
\end{equation}

Set
\begin{displaymath}
\mathbf{B}:=\{(b_1,b_2,\dots,b_n)\in\mathbb{Z}^n:
b_1+b_2+\cdots+b_n\in 2\mathbb{Z}\}.
\end{displaymath}
For $\mathbf{a}=(a_1,a_2,\dots,a_n)\in\mathbb{C}^n$ define $N(\mathbf{a})$
to be the linear span of
\begin{displaymath}
\{\mathbf{x}^{\mathbf{b}}:=
x_1^{a_1+b_1}x_2^{a_2+b_2}\cdots x_n^{a_n+b_n}:
\mathbf{b}\in\mathbf{B}
\}.
\end{displaymath}
We first define an action of the elements from $\mathbf{C}$
on $N(\mathbf{a})$ using the formulas from \eqref{eq00} as follows:
\begin{equation}\label{eq0}
\begin{array}{lcll}
X_{\varepsilon_i-\varepsilon_j}\mathbf{x}^{\mathbf{b}}&=&
(a_j+b_j)\mathbf{x}^{\mathbf{b}+\varepsilon_i-\varepsilon_j},&
1\leq i\neq j\leq n;\\
X_{\varepsilon_i+\varepsilon_j}\mathbf{x}^{\mathbf{b}}&=&
\mathbf{x}^{\mathbf{b}+\varepsilon_i+\varepsilon_j},&
i,j=1,2,\dots,n;\\
X_{-\varepsilon_i-\varepsilon_j}\mathbf{x}^{\mathbf{b}}&=&
(a_i+b_i)(a_j+b_j)\mathbf{x}^{\mathbf{b}-\varepsilon_i-\varepsilon_j},&
1\leq i\neq j\leq n;\\
X_{-2\varepsilon_i}\mathbf{x}^{\mathbf{b}}&=&
(a_i+b_i)(a_i+b_i-1)\mathbf{x}^{\mathbf{b}-2\varepsilon_i},&
i=1,2,\dots,n;\\
H_{\varepsilon_{i+1}-\varepsilon_i}\mathbf{x}^{\mathbf{b}}&=&
(a_{i+1}+b_{i+1}-a_{i}-b_{i})\mathbf{x}^{\mathbf{b}},&
i=1,2,\dots,n-1;\\
H_{2\varepsilon_1}\mathbf{x}^{\mathbf{b}}&=&
\frac{1}{2}(2a_1+2b_1+1)\mathbf{x}^{\mathbf{b}}.
\end{array}
\end{equation}

\begin{theorem}[\cite{BL}]\label{thm3}
\begin{enumerate}[(i)]
\item\label{thm3.1}
For every $\mathbf{a}\in\mathbb{C}^n$ formulae \eqref{eq0} define on
$N(\mathbf{a})$ the structure of a completely pointed weight
$\mathfrak{g}$-module.
\item\label{thm3.2}
If  $a_i\not\in\mathbb{Z}$ for all $i=1,\dots,n$, then the module
$N(\mathbf{a})$ is simple and cuspidal.
\item\label{thm3.3}
Every completely pointed simple cuspidal $\mathfrak{g}$-module is
isomorphic to $N(\mathbf{a})$ for some $\mathbf{a}\in\mathbb{C}^n$
such that $a_i\not\in\mathbb{Z}$, $i=1,\dots,n$.
\end{enumerate}
\end{theorem}

\section{The functor $\mathrm{F}$}\label{s2}

This section is similar to \cite[Subsection~3.1]{MS}.
Fix $\mathbf{a}\in\mathbb{C}^n$ such that $a_i\not\in\mathbb{Z}$,
$i=1,\dots,n$. Let $\hat{\mathcal{C}}_{\mathbf{a}}$ denote the block of
$\hat{\mathcal{C}}$ containing $N(\mathbf{a})$.
The category $\hat{\mathcal{C}}_{\mathbf{a}}$ is closed under extensions.
Denote by $\mathbb{C}[[t_1,t_2,\dots,t_n]]\text{-}\mathrm{mod}$
the category of finite dimensional
$\mathbb{C}[[t_1,t_2,\dots,t_n]]$-modules. For
$V\in \mathbb{C}[[t_1,t_2,\dots,t_n]]\text{-}\mathrm{mod}$
denote by $T_i$ the linear operator describing the action of
$t_i$ on $V$. Set $\mathbf{0}=(0,0,\dots,0)\in \mathbf{B}$.

For $\mathbf{b}\in\mathbf{B}$ consider a copy $V^{\mathbf{b}}$ of
$V$. Define
\begin{displaymath}
\mathrm{F}V:=\bigoplus_{\mathbf{b}\in\mathbf{B}} V^{\mathbf{b}}.
\end{displaymath}
Define the action of elements from $\mathbf{C}$ on the vector space
$\mathrm{F}V$ in the following way: for $v\in V^{\mathbf{b}}$ set
\begin{equation}\label{eq1}
\begin{array}{lclcl}
X_{\varepsilon_i-\varepsilon_j}v&=&
(T_j+(a_j+b_j)\mathrm{Id}_V)v&\in& V^{\mathbf{b}+
\varepsilon_i-\varepsilon_j};\\
X_{\varepsilon_i+\varepsilon_j}v&=&
v&\in& V^{\mathbf{b}+\varepsilon_i+\varepsilon_j};\\
X_{-\varepsilon_i-\varepsilon_j}v&=&
(T_i+(a_i+b_i)\mathrm{Id}_V)(T_j+(a_j+b_j)\mathrm{Id}_V)
v&\in& V^{\mathbf{b}-\varepsilon_i-\varepsilon_j};\\
X_{2\varepsilon_i}v&=&
(T_i+(a_i+b_i)\mathrm{Id}_V)(T_i+(a_i+b_i-1)\mathrm{Id}_V)
v&\in& V^{\mathbf{b}-2\varepsilon_i};\\
H_{\varepsilon_{i+1}-\varepsilon_i}v&=&
(T_{i+1}-T_i+(a_{i+1}+b_{i+1}-a_i-b_i)\mathrm{Id}_V)
v&\in& V^{\mathbf{b}};\\
H_{2\varepsilon_1}v&=&
\frac{1}{2}(2T_1+(2a_1+2b_1+1)\mathrm{Id}_V)
v&\in& V^{\mathbf{b}};
\end{array}
\end{equation}
where $i$ and $j$ are as in the respective row of \eqref{eq0}.
For a homomorphism $f:V\to W$ of $\mathbb{C}[[t_1,t_2,\dots,t_n]]$-modules
denote by $\mathrm{F}f$ the diagonally extended
linear map from $\mathrm{F}V$ to $\mathrm{F}W$, i.e.
for every $\mathbf{b}\in\mathbf{B}$ and $v\in V^{\mathbf{b}}$ set
\begin{equation}\label{eq2}
\mathrm{F}f(v)=f(v)\in W^{\mathbf{b}}.
\end{equation}

\begin{proposition}\label{prop4}
\begin{enumerate}[(i)]
\item\label{prop4.1} Formulae \eqref{eq1} define on $\mathrm{F}V$
the structure of  a $\mathfrak{g}$-module.
\item\label{prop4.15} Every $V^{\mathbf{b}}$ is a generalized
weight space of $\mathrm{F}V$. Moreover, for
$\mathbf{b}\neq \mathbf{b}'$ the weights of $V^{\mathbf{b}}$
and $V^{\mathbf{b}'}$ are different.
\item\label{prop4.2} The module $\mathrm{F}V$ belongs to
$\hat{\mathcal{C}}_{\mathbf{a}}$.
\item\label{prop4.3} Formulas \eqref{eq1} and \eqref{eq2}
turn $\mathrm{F}$ into a
functor
\begin{displaymath}
\mathrm{F}: \mathbb{C}[[t_1,t_2,\dots,t_n]]\text{-}\mathrm{mod}
\rightarrow\hat{\mathcal{C}}_{\mathbf{a}}.
\end{displaymath}
\item\label{prop4.4} The functor $\mathrm{F}$ is exact, faithful and full.
\end{enumerate}
\end{proposition}

\begin{proof}
Consider the $\mathfrak{g}$-module $N(\mathbf{a})$ for  $\mathbf{a}$
as above. Then, for every $\mathbf{b}$ the defining relations of
$\mathfrak{g}$ (in terms of elements from $\mathbf{C}$),
applied to $\mathbf{x}^{\mathbf{b}}$, can be written as
some polynomial equations in the $a_i$'s. Since \eqref{eq0} defines a
$\mathfrak{g}$-module for any $\mathbf{a}$ (Theorem~\ref{thm3}\eqref{thm3.1}),
these equations hold for any $\mathbf{a}$, that is they are actually
formal identities in the $a_i$'s.
Write now  $T_j+(a_j+b_j)\mathrm{Id}_V=A_j+B_j$,
a sum of matrices, where $A_j=T_j+a_j\mathrm{Id}_V$ and
$B_j=b_j\mathrm{Id}_V$. Note that all $A_i$ and $B_j$ commute
with each other and with all $T_l$'s. For a fixed $\mathbf{b}$,
the defining relations for $\mathfrak{g}$ on $\mathrm{F}V$  reduce
to our formal identities (in the $A_i$'s) and hence are satisfied.
This proves  claim \eqref{prop4.1}. Claim \eqref{prop4.15} follows
from the the last two lines in \eqref{eq1} and the fact
that all $T_i$'s are nilpotent (hence zero is the only eigenvalue).

As $f$ commutes with all $T_i$, the map $\mathrm{F}f$ commutes with
the action of all elements from $\mathbf{C}$ and hence defines
a homomorphism of $\mathfrak{g}$-modules. By construction we also have
$\mathrm{F}(f\circ f')=\mathrm{F}f\circ \mathrm{F}f'$, which implies
claim \eqref{prop4.3}.

By construction, $\mathrm{F}$ is exact and faithful.
It sends the simple one-dimensional
$\mathbb{C}[[t_1,t_2,\dots,t_n]]$-module to $N(\mathbf{a})$
(as in this case all $T_i=0$ and hence \eqref{eq1} gives \eqref{eq0}),
which is an object of the category $\hat{\mathcal{C}}_{\mathbf{a}}$
closed under extensions. Claim \eqref{prop4.2} follows.

To complete the proof of claim \eqref{prop4.4} we are left to show
that $\mathrm{F}$ is full. Let $\varphi:\mathrm{F}V\to
\mathrm{F}W$ be a $\mathfrak{g}$-homomorphism. Then $\varphi$ commutes
with the action of all elements from $\mathfrak{h}$. Using
claim \eqref{prop4.15}, we get that $\varphi$ induces, by restriction,
a linear map $f:V=V^{\mathbf{0}}\to W^{\mathbf{0}}=W$.
As $\varphi$ commutes with all $H_{\varepsilon_{i+1}-\varepsilon_{i}}$,
the map $f$ commutes with all operators $T_{i+1}-T_{i}$. As $\varphi$
commutes with $H_{2\varepsilon_1}$, the map $f$ commutes with $T_1$.
It follows that $f$ is a homomorphism of
$\mathbb{C}[[t_1,t_2,\dots,t_n]]$-modules.
This yields $\varphi=\mathrm{F}f$ and thus the functor $\mathrm{F}$ is full.
This completes the proof of claim \eqref{prop4.4} and of the whole
proposition.
\end{proof}

\section{Proof of Theorem~\ref{thm1}}\label{s3}

Because of Lemma~\ref{lem2} it is enough to fix one particular
block and show there that $\mathrm{F}$ is an equivalence.
Thus, we may assume that $a_i+a_j\not\in\mathbb{Z}$
for all $i,j$ (in particular, $a_i\not\in\mathbb{Z}$ for all $i$).
After Proposition~\ref{prop4}, we are only left to
show that $\mathrm{F}$ is dense (i.e. essentially surjective).
We establish density of $\mathrm{F}$ by induction on $n$.
We first prove the induction step
and then the basis of the induction, which is the case $n=2$.

Denote by $\lambda$ the weight of $\mathbf{x}^{\mathbf{0}}\in
N(\mathbf{a})$ (see Proposition~\ref{prop4}\eqref{prop4.15}). Let
$M\in \hat{\mathcal{C}}_{\mathbf{a}}$.
Set $V:=M_{\lambda}$ and denote by $M'$ the
$\mathfrak{a}$-module $U(\mathfrak{a})V$.

\subsection{Reduction to the case $n=2$}\label{s3.1}

The main result of this subsection is the following:

\begin{proposition}\label{propind}
If the functor $\mathrm{F}$ is dense for $n=2$, then it is dense
for any $n\geq 2$.
\end{proposition}

\begin{proof}
Assume that $n>2$ and that the functor $\mathrm{F}$ is dense in the case
of the algebra $\mathfrak{sp}_{2n-2}$. We realize $\mathfrak{sp}_{2n-2}$
as the subalgebra $\mathfrak{a}$ of $\mathfrak{g}$ corresponding to
the subset $\mathbf{H}_{n-1}\subset \mathbf{H}$ of simple roots.

Let $Y_1$, $Y_2$,\dots, $Y_n$ be the linear
operators representing the action of the elements $H_{2\varepsilon_1},
H_{\varepsilon_2-\varepsilon_1}$, $H_{\varepsilon_3-\varepsilon_2}$,\dots,
$H_{\varepsilon_n-\varepsilon_{n-1}}$ on $V$, respectively.
Set
\begin{equation}\label{eq3}
\begin{array}{lcl}
T_1&:=&Y_1-\frac{1}{2}(2a_1+1)\mathrm{Id}_V;\\
T_2&:=&Y_2+T_1-(a_2-a_1)\mathrm{Id}_V;\\
T_3&:=&Y_3+T_2-(a_3-a_2)\mathrm{Id}_V;\\
&&\vdots\\
T_n&:=&Y_n+T_{n-1}-(a_n-a_{n-1})\mathrm{Id}_V.
\end{array}
\end{equation}
The $T_i$'s are obviously pairwise commuting nilpotent
linear operators.

The module $M'$ is a cuspidal
generalized weight $\mathfrak{a}$-module with finite-dimensional
weight spaces. Moreover, as all composition subquotients of $M$
are of the form $N(\mathbf{a})$, all composition subquotients of
$M'$ are of the form $N(\mathbf{a})'$, the latter being a completely
pointed simple cuspidal $\mathfrak{a}$-module. By our inductive assumption,
the functor $\mathrm{F}$ is dense in the case of the algebra
$\mathfrak{a}$. Hence $M'\cong N':=\oplus_{\mathbf{b}}V^{\mathbf{b}}$,
where $\mathbf{b}\in\mathbf{B}$ is such that $b_n=0$, and the action of
$\mathfrak{a}$ on $N'$ is given by \eqref{eq1}.

\begin{lemma}\label{lemprop5}
There is a unique  (up to isomorphism) $\mathfrak{g}$-module
$Q\in\hat{\mathcal{C}}_{\mathbf{a}}$ such that $Q'=N'$ and which
gives the linear operator $T_n$ when computed using \eqref{eq3}.
\end{lemma}

\begin{proof}
The existence statement is clear, so we need only to show uniqueness.
Assume that $Q\in\hat{\mathcal{C}}_{\mathbf{a}}$ is such that $Q'=N'$
and the formulae \eqref{eq3}, applied to $Q$, produce the linear operator
$T_n$. Since $a_n\not\in\mathbb{Z}$, the endomorphism
$T_n+(a_n+b_n)\mathrm{Id}_V$ is invertible for all $b_n\in\mathbb{Z}$.
As the action of $X_{\varepsilon_n-\varepsilon_{n-1}}$ on $Q$
is bijective, we can fix a weight basis in $Q$ such that both the
$\mathfrak{a}$-action on $Q'=N'$ and the action of
$X_{\varepsilon_n-\varepsilon_{n-1}}$ on the whole $Q$ is given by
\eqref{eq1}. As $n>2$, the elements $X_{\pm 2\varepsilon_1}$ commute
with $X_{\varepsilon_n-\varepsilon_{n-1}}$ and hence their action
extends uniquely to the whole of $Q$ using this commutativity.
Similarly for all elements $X_{\pm(\varepsilon_i-\varepsilon_{i-1})}$,
$i<n-1$, and for the element $X_{\varepsilon_{n-2}-\varepsilon_{n-1}}$.
This leaves us with the elements $X_{\varepsilon_{n-1}-\varepsilon_{n-2}}$
and $X_{\varepsilon_{n-1}-\varepsilon_{n}}$. Note that the simple roots
$\varepsilon_{n-1}-\varepsilon_{n-2}$ and $\varepsilon_{n}-\varepsilon_{n-1}$
corresponding to the elements $X_{\varepsilon_{n-1}-\varepsilon_{n-2}}$
and $X_{\varepsilon_n-\varepsilon_{n-1}}$ generate a root system of type
$A_2$ (this corresponds to the algebra $\mathfrak{sl}_3$).
Therefore the fact that the action of
$X_{\varepsilon_{n-1}-\varepsilon_{n-2}}$ extends uniquely to $Q$ is
proved in \cite[Lemma~21]{MS}, and the fact that the action of
$X_{\varepsilon_{n-1}-\varepsilon_{n}}$ extends uniquely to $Q$ is proved
in \cite[Lemma~22]{MS}. This completes the proof.
\end{proof}

The module $\mathrm{F}V$ obviously satisfies $(\mathrm{F}V)'=N'$
and defines the linear operator $T_n$ when computed using \eqref{eq3}.
Hence Lemma~\ref{lemprop5} implies $M\cong \mathrm{F}V$.
Since $M\in \hat{\mathcal{C}}_{\mathbf{a}}$ was arbitrary, this
shows that the functor $\mathrm{F}$ is dense, completing the proof.
\end{proof}

\subsection{Base of the induction: some $\mathfrak{sl}_2$-theory
as preparation}\label{s3.15}

In this subsection we will recall (and slightly improve) some classical
$\mathfrak{sl}_2$-theory. We refer the reader to \cite{Maz} for
more details. Consider the Lie algebra $\mathfrak{sl}_2=
\mathfrak{sl}_2(\mathbb{C})$ with standard basis
\begin{displaymath}
\mathbf{e}:=\left(\begin{array}{cc}0&1\\0&0\end{array}\right),\quad
\mathbf{f}:=\left(\begin{array}{cc}0&0\\1&0\end{array}\right),\quad
\mathbf{h}:=\left(\begin{array}{cc}1&0\\0&-1\end{array}\right).
\end{displaymath}

Let $V$ be a finite-dimensional vector space and
$A$ and $B$ be two commuting linear operators on $V$. For
$i\in \mathbb{Z}$ denote by $V^{(i)}$ a copy of $V$ and consider
the vector space $\overline{V}:=\oplus_{i\in\mathbb{Z}}V^{(i)}$
(a direct sum of copies of $V$ indexed by $i$).
Define the actions of $\mathbf{e}$, $\mathbf{f}$ and $\mathbf{h}$
on $\overline{V}$ as follows: for $v\in V^{(i)}$ set
\begin{equation}\label{eqsl2}
\begin{array}{rclcl}
\mathbf{e}v&:=&(P-i\mathrm{Id}_V)v&\in& V^{(i+1)}\\
\mathbf{f}v&:=&(Q+i\mathrm{Id}_V)v&\in& V^{(i-1)}\\
\mathbf{h}v&:=&(Q-P+2i\mathrm{Id}_V)v&\in& V^{(i)}.
\end{array}
\end{equation}
This can be depicted as follows (here right arrows  represent the
action of $\mathbf{e}$, left arrows  represent the
action of $\mathbf{f}$ and loops  represent the
action of $\mathbf{h}$):
\begin{displaymath}
\xymatrix{
\dots\ar@/^/[rr]^{P+2\mathrm{Id}_V}&&
V^{(-1)}\ar@/^/[rr]^{P+\mathrm{Id}_V}\ar@/^/[ll]^{Q-\mathrm{Id}_V}
\ar@(dl,dr)[]_{Q-P-2\mathrm{Id}_V}
&&V^{(0)}\ar@/^/[ll]^{Q}\ar@/^/[rr]^{P}\ar@(dl,dr)[]_{Q-P}
&&V^{(1)}\ar@/^/[rr]^{P-\mathrm{Id}_V}\ar@/^/[ll]^{Q+\mathrm{Id}_V}
\ar@(dl,dr)[]_{Q-P+2\mathrm{Id}_V}&&\dots\ar@/^/[ll]^{Q+2\mathrm{Id}_V}
}
\end{displaymath}

\begin{proposition}\label{propsl2}
\begin{enumerate}[(i)]
\item\label{propsl2.1} Formulae \eqref{eqsl2} define on
$\overline{V}$ the structure of a generalized weight
$\mathfrak{sl}_2$-module with finite dimensional generalized weight
spaces.
\item\label{propsl2.2} Every cuspidal generalized weight
$\mathfrak{sl}_2$-module with finite dimensional generalized
weight spaces is isomorphic to $\overline{V}$ for some
$V$ with $P$ and $Q$ as above.
\item\label{propsl2.3} The action of the Casimir element
$\mathbf{c}:=(\mathbf{h}+1)^2+4\mathbf{f}\mathbf{e}$ on
$\overline{V}$ is given by the linear operator $(P+Q+\mathrm{Id}_V)^2$.
\item\label{propsl2.4} Let $\mathbb{C}^2$ denote the natural
$\mathfrak{sl}_2$-module (the unique two-dimensional simple
$\mathfrak{sl}_2$-module). Then the linear operator
$(\mathbf{c}-(P+Q+2\mathrm{Id}_V)^2)(\mathbf{c}-(P+Q)^2)$
annihilates the $\mathfrak{sl}_2$-module $\mathbb{C}^2\otimes \overline{V}$.
\item\label{propsl2.5} Let $\mathbb{C}^3$ denote the unique
three-dimensional simple $\mathfrak{sl}_2$-module.
Then the linear operator
$(\mathbf{c}-(P+Q+3\mathrm{Id}_V)^2)
(\mathbf{c}-(P+Q+\mathrm{Id}_V)^2)
(\mathbf{c}-(P+Q-\mathrm{Id}_V)^2)$ annihilates the
$\mathfrak{sl}_2$-module $\mathbb{C}^3\otimes \overline{V}$.
\end{enumerate}
\end{proposition}

\begin{proof}
The fact that $\overline{V}$ is an $\mathfrak{sl}_2$-module is checked by
a direct computation. That $\overline{V}$ is a generalized weight
module follows from the fact that the action of $\mathbf{h}$ on
$\overline{V}$ preserves (by \eqref{eqsl2}) each $V^{i}$ and hence is
locally finite. Since the category of generalized weight modules is
closed under extensions, to prove that $\overline{V}$ has finite
dimensional generalized weight spaces it is enough to consider the
case when $\mathbf{h}$ has a unique eigenvalue on $V^{(0)}$, say $\lambda$.
However, in this case $\mathbf{h}$ has a unique eigenvalue on $V^{i}$,
namely $\lambda+2i$, which implies that $\overline{V}^{\lambda}=V$
is finite dimensional. Claim \eqref{propsl2.1} follows. To prove
Claim \eqref{propsl2.3} we observe that the action of
$\mathbf{c}$ on $V^{i}$ is given by:
\begin{multline*}
(Q-P+(2i+1) \mathrm{Id}_V)^2+4(Q+(i+1)\mathrm{Id}_V)(P-i\mathrm{Id}_V)
=(P+Q+\mathrm{Id}_V)^2.
\end{multline*}
Claim  \eqref{propsl2.2} can be found
with all details in \cite[Chapter~3]{Maz}.

To prove claim \eqref{propsl2.4} choose a basis
$\{v_1,\dots,v_k\}$ in $V$, which gives
rise to a basis $\{v^{(i)}_1,\dots,v^{(i)}_k,i\in\mathbb{Z}\}$
in $\overline{V}$. Choose the standard
basis  $\{e_1,e_2\}$ in $\mathbb{C}^2$. Since $\mathbf{h}e_1=e_1$,
$\mathbf{h}e_2=-e_2$ and $\mathbf{h}$ acts  by $Q-P+2i\mathrm{Id}_V$
on $V^{(i)}$, we obtain that $\mathbf{h}$ acts  by
$Q-P+(2i+1)\mathrm{Id}_V$ on the vector space $W^{(i)}$ with basis
\begin{displaymath}
\{e_1\otimes v^{(i)}_1,\dots,e_1\otimes v^{(i)}_1,
e_2\otimes v^{(i+1)}_1,\dots,e_2\otimes v^{(i+1)}_1\}.
\end{displaymath}
We have $\mathbb{C}^2\otimes \overline{V}\cong \oplus_{i\in\mathbb{Z}}
W^{(i)}$ and one easily computes that in the above basis
the actions of $\mathbf{e}$ and $\mathbf{f}$ on
$\mathbb{C}^2\otimes \overline{V}$ is given by the following picture:
\begin{displaymath}
\xymatrix{
\dots\ar@/^/[rr]&&W^{(-1)}\ar@/^/[rr]^{\text{\tiny$
\left(\begin{array}{cc}P\text{+}\mathrm{Id}&\mathrm{Id}
\\0&P\end{array}\right)$}}\ar@/^/[ll]
&&W^{(0)}\ar@/^/[ll]^{\text{\tiny$\left(\begin{array}{cc}Q&0
\\\mathrm{Id}&Q\text{+}\mathrm{Id}\end{array}\right)$}}
\ar@/^/[rr]^{\text{\tiny$\left(\begin{array}{cc}P&\mathrm{Id}
\\0&P\text{-}\mathrm{Id}\end{array}\right)$}}
&&W^{(1)}\ar@/^/[rr]\ar@/^/[ll]^{\text{\tiny$\left(
\begin{array}{cc}Q\text{+}\mathrm{Id}&0
\\\mathrm{Id}&Q\text{+}2\mathrm{Id}\end{array}\right)$}}
&&\dots\ar@/^/[ll]
}
\end{displaymath}
The action of $\mathbf{c}$ on $W^{(0)}$ is now easily computed to be
given by the linear operator
{\tiny
\begin{displaymath}
G:= \left(
\begin{array}{cc}
(Q-P+2\mathrm{Id})^2+4(Q+\mathrm{Id})P& 4(Q+\mathrm{Id})\\
4P& (Q-P+2\mathrm{Id})^2+4(Q+2\mathrm{Id})(P-\mathrm{Id})+4\mathrm{Id}
\end{array}\right).
\end{displaymath}
}
The characteristic polynomial of $G$ is
\begin{displaymath}
\chi_G(\lambda)=(\lambda-(P+Q+2\mathrm{Id})^2)(\lambda-(P+Q)^2).
\end{displaymath}
Claim \eqref{propsl2.4} now follows from the Cayley-Hamilton theorem.

We have an isomorphism of $\mathfrak{sl}_2$-modules as follows:
$\mathbb{C}^2\otimes \mathbb{C}^2\cong \mathbb{C}^3\oplus \mathbb{C}$
(here $\mathbb{C}$ is the trivial module), and hence claim \eqref{propsl2.5}
follows applying claim \eqref{propsl2.4} twice. Alternatively, one could
do a direct calculation (similar to the proof of
\eqref{propsl2.3}). The proposition follows.
\end{proof}

We note that the statement of Proposition~\ref{propsl2}\eqref{propsl2.2}
is a special case of a more general result of Gabriel and Drozd describing
blocks of the category of (generalized) weight $\mathfrak{sl}_2$-modules,
in particular, simple weight $\mathfrak{sl}_2$-modules
(see \cite[7.8.16]{Di} and \cite{Dr}). The statements of Proposition~\ref{propsl2}\eqref{propsl2.4}
and \eqref{propsl2.5} are $\mathfrak{sl}_2$-refinements of a theorem
of Kostant describing possible (generalized) central characters of the
tensor product of a finite dimensional module with an infinite dimensional
module (\cite[Theorem~5.1]{Ko}).

\subsection{The case $n=2$}\label{s3.2}

Assume now that $n=2$. We have $a_1,a_2,a_1+a_2\not\in\mathbb{Z}$.
Let $\mathfrak{a}$ denote the Lie subalgebra of $\mathfrak{g}$
generated by $X_{\pm(\varepsilon_2-\varepsilon_1)}$. The algebra
$\mathfrak{a}$ is isomorphic to $\mathfrak{sl}_2$.

Let $M\in\hat{\mathcal{C}}_{\mathbf{a}}$. Denote by $\lambda$ the weight
of $\mathbf{x}^{\mathbf{0}}\in N(\mathbf{a})$ and set $V:=M_{\lambda}$.
Let $Y_1$ and $Y_2$ be the linear operators representing the actions of
the elements $H_{\varepsilon_2-\varepsilon_1}$ and
$C:=(H_{\varepsilon_2-\varepsilon_1}+1)^2+
4X_{\varepsilon_1-\varepsilon_2}X_{\varepsilon_2-\varepsilon_1}$
on $V$. The element $C$ is a Casimir element for $\mathfrak{a}$,
in particular, the operators $Y_1$ and $Y_2$ commute.
Our first observation is the following:

\begin{lemma}\label{lem7}
The action of $C$ on $V$ is invertible and hence has a square root.
\end{lemma}

\begin{proof}
From \eqref{eq0} we have that $C$ acts on $\mathbf{x}^{\mathbf{0}}$
by
\begin{displaymath}
(a_2-a_1+1)^2+4(a_2+1)a_1=(a_1+a_2+1)^2.
\end{displaymath}
Since $a_1+a_2\not\in\mathbb{Z}$ by our assumptions,
$\mathbf{x}^{\mathbf{0}}$ is an eigenvector of $C$ with a nonzero
eigenvalue. As the module $M$ has a composition series with subquotients
isomorphic to $N(\mathbf{a})$, the complex number $(a_1+a_2+1)^2\neq 0$
is the only eigenvalue of $C$ on $V$. The claim follows.
\end{proof}

Consider the $\mathfrak{a}$-module $M':=U(\mathfrak{a})M_{\lambda}$.
Let $Y'_2$ denote any square root of $Y_2$, which is a polynomial in
$Y_2$ (it exists by Lemma~\ref{lem7}). Then $Y'_2$ commutes with $Y_1$.
Set
\begin{displaymath}
T_1:=\frac{Y'_2-Y_1-\mathrm{Id}_V}{2}-a_1\mathrm{Id}_V,\quad
T_2:=\frac{Y'_2+Y_1-\mathrm{Id}_V}{2}-a_2\mathrm{Id}_V.
\end{displaymath}
Then $T_1$ and $T_2$ are two commuting nilpotent linear operators
(it is easy to check that $0$ is the unique eigenvalue for both
$T_1$ and $T_2$), hence define on $V$ the structure of
a $\mathbb{C}[[t_1,t_2]]$-module. The aim of this subsection is to
establish an isomorphism $\mathrm{F}V\cong M$, which
would complete the proof of Theorem~\ref{thm1}.

Set $R':=U(\mathfrak{a})(\mathrm{F}V)_{\lambda}$.
A direct computation (using \eqref{eq1}) shows that
$H_{\varepsilon_2-\varepsilon_1}$ and $C$ act on
$(\mathrm{F}V)_{\lambda}=V^{\mathbf{0}}$ as the linear operators $Y_1$
and $Y_2$, respectively. As any cuspidal generalized weight
$\mathfrak{a}$-module is uniquely determined by the actions of
$H_{\varepsilon_2-\varepsilon_1}$ and $C$ (see \cite{Dr} or
\cite[3.7]{Maz} for full details), it follows that $M'\cong R'$.
The isomorphism $\mathrm{F}V\cong M$ now follows from the
following statement:

\begin{proposition}
There is at most one (up to isomorphism) $\mathfrak{g}$-module
$R\in\hat{\mathcal{C}}_{\mathbf{a}}$ such that
$U(\mathfrak{a})R_{\lambda}=R'$.
\end{proposition}

\begin{proof}
Let $R\in\hat{\mathcal{C}}_{\mathbf{a}}$ be such that
$U(\mathfrak{a})R_{\lambda}=R'$. We choose a weight basis in
$R$ such that the action of $\mathfrak{a}$ on $R'$ and the
action of $X_{2\varepsilon_1}$ on $R$ is given by \eqref{eq1}
(in other words these actions coincide with the
corresponding actions on $\mathrm{F}V$).
Since $X_{\varepsilon_1-\varepsilon_2}$ commutes with
$X_{2\varepsilon_1}$, it follows that the action of
$X_{\varepsilon_1-\varepsilon_2}$ on $R$ is also given by \eqref{eq1}.

It is left to show that the action of $X_{\varepsilon_2-\varepsilon_1}$
extends uniquely from $R'$ to $R$ and then
that there is a unique way to define the action of
$X_{-2\varepsilon_1}$. This will be done
in the Lemmata~\ref{lem1002} and \ref{lem1003} below.
\end{proof}

\begin{lemma}\label{lem1002}
There is a unique way to extend the action of
$X_{\varepsilon_2-\varepsilon_1}$ from $R'$ to $R$.
\end{lemma}

\begin{proof}
Let us first show that for every $k\in\{1,2,\dots\}$ the
action of $X_{\varepsilon_2-\varepsilon_1}$ extends uniquely from
$X_{2\varepsilon_1}^{k-1}R'$ to $X_{2\varepsilon_1}^kR'$
(here $X_{2\varepsilon_1}^0R'=R'$).

Consider the following picture:
\begin{equation}\label{eq11n}
\xymatrix{
\bullet\ar@/^/@{-->}[rr]^{X} &&
\bullet\ar@/^/[ll]^{Q}\\ \\
\bullet\ar@/^/[uu]^{1}\ar@/^/[rr]^{P+1} &&
\bullet\ar@/^/[uu]^{1}\ar@/^/[ll]^{Q}\ar@/^/[rr]^{P}&&
\bullet\ar@/^/[ll]^{Q+1}
\\
}
\end{equation}
Here bullets are weight spaces with some fixed bases.
The lower row is a part of $X_{2\varepsilon_1}^{k-1}R'$ where
the $\mathfrak{a}$-action is already known by induction.
The bases in the weight spaces in the lower row are chosen such
that the action of $\mathfrak{a}$ in the lower row is given by
\eqref{eq1}. The upper row is a part of
$X_{2\varepsilon_1}^{k}R'$ where the $\mathfrak{a}$-action
is to be determined. Arrows pointing up indicate the
action of $X_{2\varepsilon_1}$. The bases of the weight spaces
in the upper row are chosen such that the action of
$X_{2\varepsilon_1}$  is given by the operator
$\mathrm{Id}_V$ (as in \eqref{eq1}). Left arrows indicate the action
of $X_{\varepsilon_1-\varepsilon_2}$. The latter  commutes with the
action of $X_{2\varepsilon_1}$ and hence is given by the same linear
operator in each column. Right arrows indicate the action
of $X_{\varepsilon_2-\varepsilon_1}$ (which is known for
$X_{2\varepsilon_1}^{k-1}R'$ and is to be determined for
$X_{2\varepsilon_1}^{k}R'$). The part to be determined is given
by the dashed arrow. Labels $P$ and $Q$ represent coefficients
(which are linear operators on $V$)
appearing in the corresponding parts of formulae \eqref{eq1}.
Note that $P$ and $Q$ commute.
The action of $X_{\varepsilon_2-\varepsilon_1}$ on
$X_{2\varepsilon_1}^{k}R'$ which is to be determined is given
by some unknown linear operators $X$.

From $H_{\varepsilon_2-\varepsilon_1}=
[X_{\varepsilon_2-\varepsilon_1},X_{\varepsilon_1-\varepsilon_2}]$
we compute that the action of $H_{\varepsilon_2-\varepsilon_1}$ on the
middle weight space in the lower row  is given by $Q-P$.
Using $[H_{\varepsilon_2-\varepsilon_1},X_{2\varepsilon_1}]=
-2X_{2\varepsilon_1}$ we get that $H_{\varepsilon_2-\varepsilon_1}$
acts on the right dot of the upper row via $Q-P-2$.
Using $[H_{\varepsilon_2-\varepsilon_1},X_{\varepsilon_1-\varepsilon_2}]=
-2X_{\varepsilon_1-\varepsilon_2}$ we get that
$H_{\varepsilon_2-\varepsilon_1}$ acts on the left dot of the upper
row via $Q-P-4$.  Hence the action of $C$ on the upper row is given by
$(Q-P-3)^2+4XQ$. The action of $C$ on the lower row
is given by $(Q-P-1)^2+4(P+1)Q=(Q+P+1)^2$.

The elements $X_{2\varepsilon_1}$, $X_{2\varepsilon_2}$ and
$X_{\varepsilon_1+\varepsilon_1}$ form a weight basis of a
simple three-dimensional $\mathfrak{a}$-module $\mathbb{C}^3$
with respect to the adjoint action of $\mathfrak{a}$.
Hence the upper row of our picture is a subquotient of the
tensor product of the lower row and $\mathbb{C}^3$.
Therefore, from Proposition~\ref{propsl2}\eqref{propsl2.5} we
obtain that the linear operator
\begin{displaymath}
(C-(Q+P-1)^2)(C-(Q+P+1)^2)(C-(Q+P+3)^2)
\end{displaymath}
annihilates the upper row. A direct computation using \eqref{eq1}
shows that the action of the operators $C-(Q+P-1)^2$ and
$C-(Q+P+1)^2$ on the part $X_{2\varepsilon_1}^{k}N(\mathbf{a})'$ of
the module $N(\mathbf{a})$ is invertible. As the
$\mathfrak{g}$-module we are working with must have a composition
series with subquotients $N(\mathbf{a})$, it follows that the
action of both $C-(Q+P-1)^2$ and
$C-(Q+P+1)^2$ on $X_{2\varepsilon_1}^{k}R'$ is invertible.
Hence $C-(Q+P+3)^2$ annihilates $X_{2\varepsilon_1}^{k}R'$, which
gives us the equation
\begin{displaymath}
(Q-P-3)^2+4XQ=(Q+P+3)^2.
\end{displaymath}
This equation has a unique solution, namely $X=Q+3$, which gives the
required extension.

Similarly one shows that for $k\in\{-1,-2,\dots\}$ the
action of $X_{\varepsilon_2-\varepsilon_1}$ extends uniquely from
$X_{2\varepsilon_1}^{k+1}R'$ to $X_{2\varepsilon_1}^kR'$
(here again $X_{2\varepsilon_1}^0R'=R'$). This completes the proof of
our lemma.
\end{proof}

\begin{lemma}\label{lem1003}
There is a unique way to define the action of
$X_{-2\varepsilon_1}$ on $N$.
\end{lemma}

\begin{proof}
To determine this action of $X_{-2\varepsilon_1}$ on $N$ we consider
the following extension of the picture \eqref{eq11n} with the same notation
as in the proof of Lemma~\eqref{lem1002}:
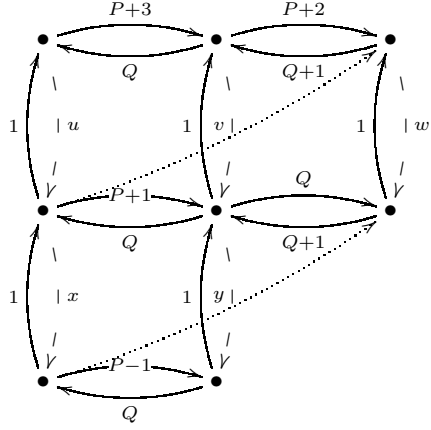
\begin{displaymath}
\xymatrix{
\bullet\ar@/^/@{-->}[dd]^{u}\ar@/^/[rr]^{P+3} &&
\bullet\ar@/^/[ll]^{Q}\ar@/^/@{-->}[dd]_{v}\ar@/^/[rr]^{P+2}&&
\bullet\ar@/^/[ll]^{Q+1}\ar@/^/@{-->}[dd]^{w}\\ \\
\bullet\ar@/^/[uu]^{1}\ar@/^/[rr]|-{P+1}\ar@/^/@{-->}[dd]^{x}
\ar@/_/@{.>}[rrrruu]&&
\bullet\ar@/^/[uu]^{1}\ar@/^/[ll]^{Q}\ar@/^/@{-->}[dd]_{y}\ar@/^/[rr]^{Q}&&
\bullet\ar@/^/[uu]^{1}\ar@/^/[ll]^{Q+1}\\ \\
\bullet\ar@/^/[uu]^{1}\ar@/^/[rr]|-{P-1}\ar@/_/@{.>}[rrrruu] &&
\bullet\ar@/^/[uu]^{1}\ar@/^/[ll]^{Q}
}
\end{displaymath}
Here all right arrows, representing the action of
$X_{\varepsilon_2-\varepsilon_1}$,
are now  determined by
Lemma~\ref{lem1002} and we have to figure out the down arrows,
representing the action of $X_{-2\varepsilon_1}$.
The two dotted arrows will be used later on in the proof.

Consider the  $\mathfrak{sl}_2$-subalgebra $\mathfrak{c}$ of
$\mathfrak{g}$ generated by $e:=X_{2\varepsilon_1}$ and
$f:=X_{-2\varepsilon_1}$. Set $h:=[e,f]$. Denote by $Z$ the action
of $h$ in the leftmost weight space of the middle row.
Then $Z=x-u$. The element $h$ commutes with both $h$ and
$H_{\varepsilon_2-\varepsilon_1}$. Therefore, by \eqref{eq1}, the
operator $Z$ commutes with both $T_1$ and $T_2$ and hence with both
$P$ and $Q$.

The algebra algebra $\mathfrak{c}$ has the quadratic Casimir element
$C_{\mathfrak{c}}$, whose action on the $\mathfrak{c}$-module given by
the leftmost column of our picture is given by $x+f(Z)$, where $f$ is
some polynomial of degree two. From \eqref{eq1} it follows that the
unique eigenvalue of this action is nonzero, in particular, $x+f(Z)$
is invertible. Let $x'$ be a fixed square root $x+f(Z)$, which is a
polynomial in $x+f(Z)$.

The elements $X_{\varepsilon_2-\varepsilon_1}$ and
$X_{\varepsilon_2+\varepsilon_1}$ form a basis of a simple
two-dimensional $\mathfrak{c}$-module with respect to the adjoint action.
Using Proposition~\ref{propsl2}\eqref{propsl2.4} and arguments
similar to those used in the proof of Lemma~\ref{lem1002}, we
get that $C_{\mathfrak{c}}-(x'+1)^2$ or
$C_{\mathfrak{c}}-(x'-1)^2$ annihilates the middle column
(the sign depends on the original choice of $x'$). Note that the middle
column equals $X_{\varepsilon_2-\varepsilon_1}$ applied to the
leftmost column.

Similarly, the elements $X_{\varepsilon_1-\varepsilon_2}$
and $X_{-\varepsilon_2-\varepsilon_1}$ form a basis of a simple
two-dimensional $\mathfrak{c}$-module with respect to the adjoint
action.  Applying the same arguments as in the previous paragraph
we get that $C_{\mathfrak{c}}-(x')^2$ annihilates any vector of the
form $X_{\varepsilon_1-\varepsilon_2}
X_{\varepsilon_2-\varepsilon_1}\mathtt{v}$,
where $\mathtt{v}$ is from the leftmost column.
This implies that the actions of $C_{\mathfrak{c}}$ and
$X_{\varepsilon_1-\varepsilon_2}X_{\varepsilon_2-\varepsilon_1}$
and thus the actions of $C_{\mathfrak{c}}$ and $C$
on the leftmost column commute.
As the action of $H$ commutes with the
action of $C$, we thus obtain that $x$ commutes with the action of
$C$. This implies that $x$ commutes with $T_1+T_2$. As
it obviously commutes with $T_1-T_2$, we get that $x$ commutes with
both $T_1$ and $T_2$ and hence with both $P$ and $Q$.

Similarly one shows that $y$, $u$, $v$ and $w$ commute with both $P$
and $Q$. From the commutativity of $X_{\varepsilon_2-\varepsilon_1}$
and $X_{-2\varepsilon_1}$ we get the following conditions:
\begin{displaymath}
y(P+1)=(P-1)x,\,\,
v(P+3)=(P+1)u,\,\,
w(P+2)(P+3)=P(P+1)u.
\end{displaymath}
Here everything commutes by the above and $P+1$, $P+2$ and $P+3$ are
invertible (as $X_{\varepsilon_2-\varepsilon_1}$ acts bijectively).
Therefore
\begin{displaymath}
y=(P-1)(P+1)^{-1}x,\,\,
v=(P+1)(P+3)^{-1}u,\,\,
w=P(P+1)(P+3)^{-1}(P+2)^{-1}u.
\end{displaymath}
This implies that $y$, $v$ and $w$ are uniquely determined by $x$ and $u$.

Since the actions of both $X_{\varepsilon_2-\varepsilon_1}$ and
$X_{2\varepsilon_1}$ are completely determined, we can compute the
action of $X_{2\varepsilon_2}$ and see that it is given
(similarly to the action of $X_{2\varepsilon_1}$) by $\mathrm{Id}_V$
(this is depicted by the dotted arrows in the picture).
As $X_{-2\varepsilon_2}$ and $X_{2\varepsilon_2}$ commute, we obtain
that $w=x$, that is
\begin{equation}\label{eq23}
x= P(P+1)(P+3)^{-1}(P+2)^{-1}u.
\end{equation}
Therefore the only parameter left for now is $u$.

On the one hand, the action of the element $h$ on the middle dot of
the second row is given by  $y-v=(P-1)(P+1)^{-1}x-(P+1)(P+3)^{-1}u$.
On the other hand, from $[h,X_{\varepsilon_2-\varepsilon_1}]=4
X_{\varepsilon_2-\varepsilon_1}$ we have that this action
equals $Z+4=x-u+4$. This gives us the equation
\begin{equation}\label{eq22}
(P-1)(P+1)^{-1}x-(P+1)(P+3)^{-1}u=x-u+4.
\end{equation}

Using \eqref{eq22} and \eqref{eq23} we get the equation
\begin{displaymath}
\frac{P(P-1)}{(P+2)(P+3)}u+
\frac{P+1}{P+3}u
=\frac{P(P+1)}{(P+2)(P+3)}u -u +4.
\end{displaymath}
This is a linear equation with nonzero coefficients
and thus it has a unique solution,
namely $u=(P+3)(P+2)$. Hence $u$ is uniquely defined.
The claim of the lemma follows.
\end{proof}

\section{Consequences}\label{s4}

\begin{corollary}\label{cor31}
Let $\mathbf{a}\in\mathbb{C}^n$ be such that $a_i\not\in\mathbb{Z}$
and $a_i+a_j\not\in\mathbb{Z}$ for all $i$ and $j$.
Let $M\in\hat{\mathcal{C}}$ and $\lambda\in\mathrm{supp}(M)$.
Denote by $U_0$ the centralizer of $\mathfrak{h}$ in $U(\mathfrak{g})$.
Then for any $A,B\in U_0$ the actions of $A$ and $B$ on
$M_{\lambda}$ commute.
\end{corollary}

\begin{proof}
By Proposition~\ref{prop4}, we may assume that
$M\cong \mathrm{F}V$. For the module $\mathrm{F}V$ the claim
follows from the formulae \eqref{eq1}.
\end{proof}

\begin{corollary}\label{cor32}
For any simple weight cuspidal $\mathfrak{g}$-module $L$
with finite dimensional weight spaces we have
$\dim\mathrm{Ext}^1_{\mathfrak{g}}(L,L)=n$.
\end{corollary}

\begin{proof}
This follows from Theorem~\ref{thm1} and the observation
that a similar equality is true for the unique simple
$\mathbb{C}[[t_1,t_2,\dots,t_n]]$-module.
\end{proof}

We also recover the main result of \cite{BKLM}:

\begin{corollary}[\cite{BKLM}]\label{cor33}
The category of all weight cuspidal $\mathfrak{g}$-modules is semi-simple.
\end{corollary}

\begin{proof}
By \cite[Lemma~2]{BKLM}, all blocks of the category of weight cuspidal $\mathfrak{g}$-modules are equivalent. Hence it is enough to prove the
claim for the block containing $N(\mathbf{a})$ for some $\mathbf{a}\in\mathbb{C}^n$ such that $a_i+a_j\not\in\mathbb{Z}$ for all $i,j$.
From \eqref{eq1} it follows that the module $\mathrm{F}V$ is weight if
and only if all operators $T_i$ are semi-simple, hence zero. Therefore
from Theorem~\ref{thm1} we get that the block of the category of
weight cuspidal modules is equivalent to the category of
finite dimensional modules over
$\mathbb{C}[[t_1,t_2,\dots,t_n]]/(t_1-0,t_2-0,\dots,t_n-0)
\cong \mathbb{C}$. The claim follows.
\end{proof}

\vspace{0.5cm}

\noindent
V.M.: Department of Mathematics, Uppsala University, SE 471 06,
Uppsala, SWEDEN, e-mail: {\tt mazor\symbol{64}math.uu.se}
\vspace{0.2cm}

\noindent
C.S.: Mathematik Zentrum, Universit{\"a}t Bonn,
Endenicher Allee 60, D-53115, Bonn, GERMANY,
e-mail: {\tt stroppel\symbol{64}uni-bonn.de}
\vspace{0.2cm}


\begin{thebibliography}{9999999}
\bibitem[BKLM]{BKLM} D.~Britten, O.~Khomenko, F.~Lemire, V.~Mazorchuk;
Complete reducibility of torsion free $C_n$-modules of finite degree.
J. Algebra  {\bf 276}  (2004),  no. 1, 129--142.
\bibitem[BL]{BL} D.~Britten, F.~Lemire; A classification of simple
Lie modules having a $1$-dimensional weight space.  Trans. Amer. Math.
Soc. {\bf 299}  (1987),  no. 2, 683--697.
\bibitem[Di]{Di} J.~Dixmier; Enveloping algebras. Graduate Studies in
Mathematics, {\bf 11}. American Mathematical Society, Providence, RI, 1996.
\bibitem[Dr]{Dr} Yu.~Drozd; Representations of Lie algebra
${\mathfrak{sl}}(2)$. Visnyk Kyiv. Univ. Ser. Mat. Mekh.  No. {\bf 25}
(1983), 70--77.
\bibitem[Fe]{Fe} S.~Fernando; Lie algebra modules with finite-dimensional
weight spaces. I. Trans. Amer. Math. Soc. {\bf 322} (1990), no. 2, 757--781.
\bibitem[GS]{GS} D.~Grantcharov, V.~Serganova; Cuspidal representations
of $\mathfrak{sl}(n+1)$, Adv. Math doi:10.1016/j.aim.2009.12.024
\bibitem[Ko]{Ko} B.~Kostant; On the tensor product of a finite and an
infinite dimensional representation.  J. Functional Analysis  {\bf 20}
(1975), no. 4, 257--285.
\bibitem[Mat]{Ma} O.~Mathieu; Classification of irreducible weight
modules.  Ann. Inst. Fourier (Grenoble) {\bf 50}  (2000),  no. 2, 537--592.
\bibitem[Maz]{Maz}  V.~Mazorchuk; Lectures on
$\mathfrak{sl}_2(\mathbb{C})$-modules. Imperial College Press, 2009.
\bibitem[MS]{MS} V.~Mazorchuk, C.~Stroppel; Cuspidal
$\mathfrak{sl}_n$-modules and deformations of certain
Brauer tree algebras, Preprint arXiv:1001.2633.
\end{thebibliography}
\end{document}